\documentclass[12pt]{article}
\usepackage[utf8]{inputenc}
\usepackage[margin=1in]{geometry}
\usepackage{enumerate}
\usepackage{amsmath}
\usepackage{amssymb}
\usepackage{amsthm}
\usepackage{amsfonts}
\usepackage{xcolor}
\usepackage{comment}
\usepackage{hyperref}
\usepackage{cleveref}
\usepackage{cite}

\newcommand{\nl}{\newline}

\newcommand{\R}{\mathbb{R}}

\newcommand{\Z}{\mathbb{Z}}

\newcounter{custom}

\newtheorem{theo}[custom]{Theorem}
\newtheorem{coro}[custom]{Corollary}
\newtheorem{lemm}[custom]{Lemma}
\newtheorem{prop}[custom]{Proposition}

\crefname{theo}{Theorem}{Theorems}
\crefname{lemm}{Lemma}{Lemmas}
\crefname{coro}{Corollary}{Corollaries}
\crefname{prop}{Proposition}{Propositions}

\author{Ferdinand Jacobé de Naurois}

\begin{document}
	\begin{center}
		{\Large\textbf{Limits of harmonic functions on $\Z^d$}}
	\end{center}
	\begin{center}
		Ferdinand Jacobé de Naurois\\
		Mathematical Institute - University of Oxford\\
		\href{mailto:ferdinand.jacobedenaurois@maths.ox.ac.uk}{ferdinand.jacobedenaurois@maths.ox.ac.uk}
	\end{center}
	
\begingroup
	\renewcommand{\thefootnote}{}
	\footnotetext{This is the author's accepted manuscript of an article published in
		\emph{L'Enseignement Math\'ematique}, \href{https://doi.org/10.4171/LEM/1109}{doi:10.4171/LEM/1109}.
		\copyright\ Fondation L'Enseignement Math\'ematique.}
	\endgroup
\begin{abstract}
		We give an example of a sequence of positive harmonic functions on $\Z^d$, $d\geq 2$, that converges pointwise to a non-harmonic function.
\end{abstract}

Let us consider a probability measure $\mu $ on $\Z^d$. We say that a non-negative function $f :\Z^d\rightarrow [0,+\infty)$ is harmonic relative to the measure $\mu$ if and only if, for every $x_0\in\Z^d$,
$$f(x_0) = \int_{\Z^d}f(x_0+x)d\mu(x).$$

The determination of harmonic functions on groups has been the subject of a lot of investigation. See for example \cite{Black},\cite{Choq-Den} and \cite{Nel} for abelian groups, and \cite{MR131904}, \cite{Mar} and \cite{Ben1,Ben2} for the case of nilpotent groups. Recall that $\Z^d$ endowed with any adapted probability measure (that is, not supported on a proper subgroup) has the Liouville property. This means that every bounded harmonic function is constant. A proof of this fact is given in \cite[Theorem 3]{Black}. This also follows from \Cref{characterization of harmonic functions} below, which is proved in \cite{Choq-Den}.

Since bounded harmonic functions on $\Z^d$ are only the constant functions, it is natural to investigate the richer set of positive harmonic functions. These remain classifiable in many cases. In the more general setting of locally compact abelian groups, Gustave Choquet and Jacques Deny characterised them in terms of barycenters of harmonic positive characters, see Theorem 3 in \cite{Choq-Den} for a precise statement.

This led them to study the topology of the set of harmonic functions. In the corollary to Lemma 1 of the same paper, they provide a criterion for closedness of the set of harmonic functions in the topology of pointwise convergence. In \cite{Den} (see the top of the last page), published a few weeks later, Jacques Deny explicitly mentions the existence of measures for which this set is not closed. However he provides no example.

The question of the existence of such measures surfaced again decades later when Anna Erschler pointed out an issue in \cite{polymerakis2021positive}. This led the author of \cite{polymerakis2021positive} to notice that this issue already appeared in \cite{ballmann1992discretization} and \cite{ballmann2021equivariant}. Since then, the gap in \cite{ballmann1992discretization} and \cite{ballmann2021equivariant} has been addressed in \cite{ballmann2025equivariant}, whereas the issue in \cite{polymerakis2021positive} has been fixed in the latest arXiv version.

The authors assumed that the Martin boundary of a random walk on a countable set consisted entirely of harmonic functions. The Martin boundary being always compact, this would contradict the existence of measures for which the set of harmonic functions is not closed. This fact was suspected by Anna Erschler who asked Emmanuel Breuillard if he knew of any counterexample, and he passed the problem on to us.
In this article, we address the following Question.\nl

\textbf{Question.} Let $\mu$ be a given probability measure on $\Z^d$. Is it true that a limit of non-negative harmonic functions is always harmonic?\nl

Here, by limit of harmonic functions, we mean pointwise convergence, namely $f_n\to f$ if $f_n(x)\to f(x)$ for every $x\in\Z^d$. We restrict ourselves to the study of non-negative harmonic functions or -- which amounts to the same -- to harmonic functions that are bounded below. When the measure $\mu$ has finite support, pointwise convergence implies $L^1$ convergence, and the answer to our Question is therefore positive. We will show (\Cref{harmonic super exponential}) that the same holds if $\mu$ has finite exponential moments of all orders, that is $\int_{\Z^d}e^{a|x|}d\mu(x) < +\infty$ for all $a> 0$ and a norm $|\cdot |$ on $\R^d$. The answer is always positive when $d = 1$ (see \Cref{harmonic closed for Z}), but can be negative in general. In \Cref{counter example d 2} below we provide a counterexample for each $d\geq 2$.

This counterexample will involve harmonic group homomorphisms from $\Z^d$ to the multiplicative group $\R_+^* := (0,\infty)$. The group homomorphisms $\Z^d\rightarrow \R_+^*$ will be called positive characters. They are of the form $n\mapsto e^{n\cdot s}$, where $s$ is some vector in $\R^d$ and $\cdot$ is the usual dot product. This character is denoted by $\chi_s$.

Let us consider a probability measure $\mu$ on $\Z^d$. Given a sequence $(s_n)$ of elements of $\R^d$ and $s\in \R^d$, it is clear that one has pointwise convergence $\chi_{s_n}\rightarrow \chi_s$ if and only if $s_n\rightarrow s$. Moreover, $\chi_s$ is harmonic if and only if the integral $\int_{\Z^d}\chi_s(x)d\mu(x)$ equals $1$. Therefore, if we can show that the set: 
$$E_\mu^1 = \left\{s\in \R^d\;\Big|\int_{\Z^d}\chi_sd\mu = 1\right\}$$ is not closed, then we have a counterexample to our Question. In fact, the converse is also true, as stated in \Cref{stable iff E closed} below.

In order to prove this fact, we will need the following characterisation of non-negative harmonic functions on $\Z^d$ (see \cite{Choq-Den}, Theorem 3):
\begin{lemm}\label{characterization of harmonic functions}
	Let $\mu$ be an adapted probability measure on $\Z^d$. A function $f : \Z^d\rightarrow [0,+\infty)$ is harmonic relative to $\mu$ if and only if it satisfies:
	$$\forall x\in\Z^d,\;\;f(x) = \int_{s\in E_{\mu}^1}\chi_s (x)d\nu(s)$$
	for some non-negative Borel measure $\nu$ on $E_\mu^1$.
\end{lemm}
\begin{theo}\label{stable iff E closed}
	Let $d$ be a positive integer and $\mu$ be an adapted probability measure on $\Z^d$. The set of non-negative functions on $\Z^d$ that are harmonic relative to the measure $\mu$ is stable under taking pointwise limits if and only if $E_\mu^1$ is closed in $\R^d$.
\end{theo}
\begin{proof}For $x = (x_1,\cdots,x_d)\in \R^d$, let $|x| = \max_{1\leq i\leq d}|x_i|$. If $E_\mu^1$ is not closed, then for any sequence $(s_n)$ of elements of $E_\mu^1$ converging to some $s\in\R^d\backslash E_\mu^1$, the sequence $\chi_{s_n}$ of harmonic (relative to $\mu$) positive characters converges pointwise to $\chi_s$ which is not harmonic.
	
	Let us now focus on the other direction. Suppose $E_\mu^1$ is closed and let $(f_n)$ be a sequence of non-negative functions on $\Z^d$ which are harmonic relative to $\mu$. Assume the sequence $(f_n)$ converges pointwise to some non-negative function $f : \Z^d\rightarrow[0,+\infty)$. Using \Cref{characterization of harmonic functions}, one gets a sequence of non-negative Borel measures $(\nu_n)$ on $E_\mu^1$ such that for every $n$:
	$$\forall x\in\Z^d,\; f_n(x) = \int_{s\in E_\mu^1}\chi_s(x)d\nu_n(s).$$
	
	Now, $\nu_n(E_\mu^1) = \int_{s\in E_\mu^1}\chi_s(0)d\nu_n(s) = f_n(0)$ and the sequence $(f_n(0))$ converges, so $(\nu_n(E_\mu^1))$ is a bounded sequence. Since $E_\mu^1$ is a complete metric space, in order to extract a weakly convergent subsequence of $(\nu_n)$, it suffices to show that the sequence of measures $(\nu_n)$ is tight.
	
	Towards proving tightness of the sequence $(\nu_n)$, set $A_k = E_\mu^1\cap\{s\in\R^d \;|\;|s|> k \}$ for every $k>0$.
	Let $B = \{\pm e_1,\cdots,\pm e_d\}$ be the set of vectors of the standard basis of $\R^d$ and their opposites. Then for every $n\geq1$ and $k\geq 1$ we have:
	\begin{align*}
		\sum_{x\in  B}f_{n}(x) &\geq \sum_{x\in  B}\int_{A_k}\chi_s(x)d\nu_{n}(s) = \int_{A_k}\sum_{x\in B}\chi_s(x)d\nu_{n}(s)\\
		&\geq \int_{A_k}e^{{|s|}}d\nu_{n}(s) \geq e^{{k}} \nu_{{n}}(A_k)
	\end{align*}
	
	Since the sequence $(f_n(x))$ converges for every $x\in B$, the left-hand side converges as $n$ goes to $+\infty$ hence is bounded by a constant $C$ independent of $n$ and $k$. It follows that $\nu_n(A_k)\leq Ce^{-k}$ for every $n,k\geq 1$, so the sequence $(\nu_n)$ of measures is tight and we can extract a weakly convergent subsequence. Without loss of generality, assume that $(\nu_n)$ converges weakly and denote by $\nu$ the limit.
	
	We are now left with showing that for every $x\in\Z^d$ we have:
	$$f(x) = \int_{s\in E_\mu^1}\chi_s(x)d\nu(s)$$
	which would imply directly that $f$ is harmonic relative to $\mu$ due to \Cref{characterization of harmonic functions}. Pick some $x\in\Z^d$.
	
	We first prove that $f(x)\geq \int_{s\in E_\mu^1}\chi_s(x)d\nu(s)$. For every $k\geq 1$, pick $\varphi_k :E_\mu^1\to [0,1]$ a compactly supported continuous function such that ${\varphi_k}_{|E_\mu^1\backslash A_k} = 1$. Then one has for all $n,k\geq 1$: $f_n(x)\geq \int_{E_\mu^1}\varphi_k(s)\chi_s(x)d\nu_n(s)$. Upon letting $n$ go to $+\infty$, using weak convergence of $(\nu_n)$ to $\nu$ and the fact that the functions $s\mapsto \varphi_k(s)\chi_s(x)$ are continuous compactly supported, we get: $f(x)\geq \int_{E_\mu^1}\varphi_k(s)\chi_s(x)d\nu(s)$. Let $k\to+\infty$ using Fatou's lemma to conclude.
	
	We now prove the converse inequality.
	For every $n,k \geq 1$, we have by the Cauchy-Schwarz inequality:
	\begin{align*}
	\int_{A_k}\chi_s(x)d\nu_n(s) &\leq \sqrt{\int_{A_k}\chi_s(x)^2d\nu_n(s)}\sqrt{\nu_n(A_k)}\\
	&\leq \sqrt{\int_{E_\mu^1}\chi_s(2x)d\nu_n(s)}\sqrt{\nu_n(A_k)} = \sqrt{f_n(2x)}\sqrt{\nu_n(A_k)}\\
	&\leq M_x\sqrt{C}e^{-k/2}
	\end{align*}
	where $M_x$ is an upper bound for the convergent sequence $(\sqrt{f_n(2x)})$. 
	
	For all $n,k\geq 1$, we deduce from the above inequality that:
	\begin{align*}
		f_n(x) \leq \int_{E_\mu^1} \varphi_k(s)\chi_s(x)d\nu_n(s)+ M_xC^{1/2} e^{-k/2}
	\end{align*}
	The function $s\mapsto \varphi_k(s)\chi_s(x)$ is continuous and compactly supported, so upon letting $n$ go to $+\infty$ in the above and using weak convergence of $(\nu_n)$ to $\nu$, we get: $f(x) \leq  \int_{E_\mu^1} \varphi_k(s)\chi_s(x)d\nu(s)+ M_xC^{1/2} e^{-k/2} \leq \int_{E_\mu^1} \chi_s(x)d\nu(s)+ M_xC^{1/2} e^{-k/2}$. Letting $k$ go to $+\infty$ then provides $f(x) \leq \int_{E_\mu^1}\chi_s(x)d\nu(s)$ as expected, which concludes the proof.
\end{proof}
Note that this implies that if there is a sequence of non-negative harmonic functions which converges pointwise to a non-harmonic function, then it is also true that there is such a sequence made only of harmonic positive characters.

It follows from this theorem that understanding the stability of non-negative harmonic functions under taking pointwise limits amounts to studying the closure of $E_\mu^1$. As a consequence, we may apply this criterion to show that on $\Z$, a pointwise limit of non-negative harmonic functions is always harmonic.
\begin{lemm}\label{E compact on Z}
	Let $\mu$ be a  probability measure on $\Z$ such that $\mu(\{0\})<1$. Then, $E_\mu^1$ has at most two points.
\end{lemm}
\begin{proof}
	The function $s\mapsto \int_{\Z}\chi_s(x)d\mu(x)$ is strictly convex, as sum of such functions.
\end{proof}
\begin{coro}\label{harmonic closed for Z}
	Let $\mu$ be a probability measure on $\Z$. Then, every pointwise limit of non-negative harmonic functions on $\Z$ relative to the measure $\mu$ is harmonic.
\end{coro}
\begin{proof}
	The result is obvious when $\mu(\{0\}) = 1$. Use \Cref{stable iff E closed} and \Cref{E compact on Z} in the case where $\mu$ is adapted. In the general case when $\mu$ is not adapted but not reduced to a Dirac mass at $0$, denote by $H$ the subgroup generated by the support of $\mu$. A function $f : \Z\rightarrow [0,+\infty)$ is harmonic relative to $\mu$ if and only if for every $x_0\in \Z$, $f_{x_0} : x\in H\mapsto f(x+x_0)$ is harmonic on $H$ for the measure $\mu_H$, restriction of $\mu$ to $H$. Hence if $(f_n)$ is a sequence of non-negative harmonic functions relative to $\mu$ converging pointwise to some $f : \Z\rightarrow[0,+\infty)$, then the $f_{n,x_0}$ are all harmonic relative to $\mu_H$. So are the $f_{x_0}$ for all $x_0\in \Z$, due to the previous case. It follows that $f$ is harmonic relative to $\mu$.
\end{proof}
\noindent However, in the general case where $d\geq 2$, this result fails. We construct a counterexample.
\begin{theo}\label{counter example d 2}
	For every $d\geq 2$, there exists a probability measure $\mu$ on $\Z^d$, the support of which generates $\Z^d$ as a semigroup, such that $E_\mu^1$ is not closed.\nl
	In particular, for such a measure $\mu$, there exist pointwise limits of positive harmonic functions that are not harmonic.
\end{theo}	
\begin{proof}
	\noindent Let $d\geq 2$ be an integer, and $\nu = (\nu_{x})_{x\in \Z^d}$ be the positive measure on $\Z^d$ defined as follows, for every $x_1,\cdots,x_d\in \Z$:
	\begin{align*}
		\nu_{(x_1,\cdots,x_d)} &= \frac{1}{M}\frac{e^{100\cdot n-n^2}}{n^2}\;\;&\;&\text{if $(x_1,\cdots,x_d) = (100n,-n^2,0,\cdots,0)$ for some integer $n\geq 1$,}\\
		\nu_{(x_1,\cdots,x_d)} &= 0 \;\;&\;&\text{else,}
	\end{align*}
	where $M$ is defined so that the $\nu_{x}$ sum to $1$:
	$$M = \sum_{n = 1}^{+\infty}\frac{e^{100\cdot n-n^2}}{n^2}.$$
	Note that $M\geq e^{99}$.
	
	We denote by $\eta$ the uniform probability measure on the set $B = \{\pm e_i \; |\; i\in [\![1,d]\!]\}$ where $(e_i)_{1\leq i\leq d}$ is the standard basis of $\Z^d$. Finally, for some $\delta > 0$ to be determined later, we write $\mu = \delta \eta+ (1-\delta)\nu$. Of course, the support of $\mu$ generates $\Z^d$ as a semigroup: this is the point of adding the $\eta$ component. We claim that if $\delta > 0$ is small enough, then $E_\mu^1$ is not closed.
	
	Let $f : \R^2 \rightarrow [0,+\infty]$ be the function defined by the following equation:
	$$\forall s,t\in \R,\;\; f(s,t) = \int_{x\in \Z^d}\chi_{(s-1,t-1,0,\cdots,0)}(x)d\mu(x) = \frac{1-\delta}{M}\sum_{n=1}^{+\infty}\frac{e^{100\cdot ns- n^2t}}{n^2} + \delta\cdot  h(s,t).$$
	where $h(s,t) = 1 + \frac{1}{d}(\cosh(s-1)+\cosh(t-1)-2)$ is a continuous function from $\R^2$ to $\R$, independent of $\delta$. From now on, we assume that $\delta < \min(\frac{1}{2h(0,200)}, \frac{1}{4})$.
	
	We turn our attention to the point $A = (-1,-1,0,\cdots,0)$. Since $M> e^{99}$, we have $f(0,0) \leq \frac{\zeta(2)}{M}  + 2\delta< 1$. Therefore $A\notin E_\mu^1$. However, we will prove that $A$ lies in the closure of $E_\mu^1$ if $\delta>0$ has been chosen small enough.
	
	For all $s \in (0,1)$, let:
	\begin{align*}
		g_s : \R&\rightarrow[0,+\infty]\\
		t&\mapsto f(s,t).
	\end{align*}
	Let $s \in (0,1)$ be fixed for now.
	
	We have $g_s(t) = +\infty$ whenever $t \leq 0$, and $g_s$ is continuous on $(0,+\infty)$. Moreover, all terms being positive, we have $g_s(t)\xrightarrow[t\rightarrow 0^+]{} +\infty$. Moreover, an elementary calculation shows:
	$$g_s(200)< \frac{1}{M}\sum_{n=1}^{+\infty}\frac{e^{100n - n^2\cdot200}}{n^2}+ \delta \cdot h(0,200) \leq \frac{1}{M} + \delta h(0,200).$$
	Hence if $\delta > 0$ is small enough ($\delta < \frac{1}{2h(0,200)}$ suffices) we get $g_s(200)< 1$. By continuity, it follows that there exists some $t_s\in (0,200)$ such that $g_s(t_s) = 1$.
	
	The only thing left to show is that $t_s\xrightarrow[s\rightarrow 0^+]{} 0$. If this is true, then we have $(s-1,t_s-1,0,\cdots,0)\xrightarrow[s\rightarrow 0^+]{}A $ with $(s-1,t_s-1,0,\cdots,0)\in E_\mu^1$ for all $s \in (0,1)$. So $A$ is in the closure of $E_\mu^1$ but not in $E_\mu^1$.
	
	Let us now prove that $t_s\xrightarrow[s\rightarrow 0^+]{} 0$. If it is not true, then one can find $\varepsilon > 0$ and a sequence $(s_k)_{k\geq 0}$ of elements of $(0,1)$ with limit $0$ such that $t_{s_k}\geq \varepsilon$ for all $k$. We then get for all $k$:
	\begin{align*}
		1 &= g_{s_k}(t_{s_k}) \leq \frac{1}{M}\sum_{n=1}^{+\infty}\frac{e^{{100\cdot ns_k}-n^2t_{s_k}}}{n^2}+ \delta h(s_k,t_{s_k})\leq \frac{1}{M}\sum_{n=1}^{+\infty}\frac{e^{{100\cdot ns_k}-n^2\varepsilon}}{n^2}+ \delta h(0,200)\\
		&\leq \frac{1}{M}\sum_{n=1}^{+\infty}\frac{e^{{100\cdot ns_k}-n^2\varepsilon}}{n^2}+ \frac{1}{2}.
	\end{align*}
	
	For every $k$ the $n$-th summand in the above sum is bounded by $e^{100\cdot n-\varepsilon n^2}/{n^2}$ which sums to a finite value. So we can use the dominated convergence theorem and let $k\rightarrow +\infty$ to get:
	$$1\leq \frac{1}{M}\sum_{n=1}^{+\infty}\frac{e^{-n^2\varepsilon}}{n^2} + \frac{1}{2}\leq \frac{\zeta(2)}{M} + \frac{1}{2},$$
	which is a contradiction, since $M\geq e^{99}$. This concludes the proof.
\end{proof}
We show now that when the measure $\mu$ decreases sufficiently fast, then the set of positive harmonic functions is stable under limits. We define a probability measure $\mu$ on $\Z^d$ to have finite exponential moments of all orders if and only if, for some norm $|.|$ on $\R^d$, one has $\int_{\Z^d}e^{a|x|}d\mu(x) < +\infty$ for all $a>0$. This does not depend on the choice of the norm.

Note that the counterexample constructed above has a finite exponential moment (that is, $\int e^{a|x|}d\mu(x)$ converges for some $a> 0$) but does not have finite exponential moments of all orders.
\begin{theo}\label{harmonic super exponential}
	Let $d\geq 1$ and $\mu$ be a probability measure on $\Z^d$ with finite exponential moments of all orders. Then every pointwise limit of non-negative harmonic functions on $\Z^d$ relative to the measure $\mu$ is harmonic.
\end{theo}
\begin{proof}
	We start with the case where $\mu$ is adapted, that is, not supported on a proper subgroup. Let $\mu$ be an adapted probability measure on $\Z^d$ such that $\int_{\Z^d}e^{a|x|}d\mu(x) < +\infty$ for some norm $|\cdot |$ on $\R^d$ and all $a > 0$. By \Cref{stable iff E closed}, it suffices to show that $E_\mu^1$ is closed.
	
	Let $(s_n)$ be a sequence of elements of $E_\mu^1$ which converge to some $s\in\R^d$. Then by boundedness of $(s_n)$ there is a constant $C$ such that $e^{x\cdot s_n} \leq e^{C|x|}$ for all $n$ and all $x\in\Z^d$. By dominated convergence one has $s\in E_\mu^1$.
	
	Now let us assume that the measure $\mu$ is not adapted, and let $H$ be the subgroup generated by its support. Then, we can restrict $\mu$ to $H$ and get an adapted measure $\mu_H$ on the subgroup $H$. Now a function $f : \Z^d \rightarrow [0,+\infty)$ is harmonic relative to $\mu$ if and only if for every $x_0\in \Z^d$, the function $f_{x_0} : x\in H\mapsto f(x+x_0)$ is harmonic relative to $\mu_H$. We conclude as in the proof of \Cref{harmonic closed for Z}.
\end{proof}
One may also ask more precisely what the closure $\overline{E_\mu^1}$ of $E_\mu^1$ looks like compared to $E_\mu^1$, for an arbitrary probability measure $\mu$. We conclude this article with a positive result in this direction:
\begin{prop}\label{zero measure}
	Let $\mu$ be a probability measure on $\Z^d$. Then, $\overline{E_\mu^1}\backslash E_\mu^1$ has zero Lebesgue measure.
\end{prop}
With this goal in mind, we define the set $E_\mu = \{s\in \R^d \;|\;\int_{\Z^d}\chi_sd\mu < +\infty\}$. The following shows that the only points that can witness the non-closedness of $E_\mu^1$ must lie on the boundary of $E_\mu$.
\begin{prop}
	The set $\overline{E_\mu^1}\backslash E_\mu^1$ is contained in the boundary of $E_\mu$.
\end{prop}
\begin{proof}
	This is a consequence of the following \Cref{lemma closure}. Indeed, let $s$ be a point in $\overline{E_\mu^1}\backslash E_\mu^1$. If it is not in the boundary of $E_\mu$, then we must have $s\in \overline{E_\mu^1}\cap \mathring{E_\mu}$, where $\mathring{E_\mu}$ is the interior of $E_\mu$. So $s\in E_\mu^1$ by \Cref{lemma closure}, which is a contradiction.
\end{proof}
\begin{lemm}\label{lemma closure}
	Let $\mu$ be a probability measure on $\Z^d$. Then, $\overline{E_\mu^1}\cap \mathring{E_\mu}\subseteq E_\mu^1$.
\end{lemm}
\begin{proof}
	Let $(s_n)$ be a sequence of elements of $E_\mu^1$ which converge to some $s\in \mathring{E_\mu}$. Let us prove that $s\in E_\mu^1$. For this, consider some $\varepsilon > 0$ such that $s + [ -\varepsilon,\varepsilon ]^d\subseteq {E_\mu}$. Now, for every large enough positive integer $n$, $s_n\in s + [ -\varepsilon,\varepsilon ]^d$, so $s_n$ is a barycenter of the $2^d$ corners $c_1,\cdots,c_{2^d}$ of the cube $s + [ -\varepsilon,\varepsilon ]^d$. Convexity of the exponential function therefore ensures that:
	$$\chi_{s_n}\leq \sum_{i=1}^{2^d}\chi_{c_i}.$$
	
	The right-hand side of the inequality has finite integral, so the dominated convergence theorem applies and proves that $s\in E_\mu^1$.
\end{proof}
To conclude the proof of \Cref{zero measure}, note that $E_\mu$ is a convex set. Indeed, if $s,t\in \R^d$ are in $E_\mu$, then for every $\lambda\in [0,1]$ one has $\chi_{\lambda s+ (1-\lambda)t}\leq \lambda\chi_s+(1-\lambda)\chi_t$, so $\chi_{\lambda s+ (1-\lambda)t} $ has finite integral. We thus use the following well-known result (see \cite[Theorem 1]{Lang}) to conclude the proof of \Cref{zero measure}.
\begin{lemm}
	The boundary of any convex subset of $\R^d$ has zero Lebesgue measure.
\end{lemm}

	\section*{Acknowledgements}
	I would like to express my gratitude to Emmanuel Breuillard and Anna Erschler for suggesting this question. I would also like to thank the anonymous referee as well as Werner Ballmann and Panagiotis Polymerakis for their helpful comments.

	\section*{Funding}
	The author is supported by the Clarendon Fund and Keble College, University of Oxford.

\end{document}